\documentclass{article}

\PassOptionsToPackage{numbers, sort&compress}{natbib}

\usepackage{natbib}
\usepackage[preprint]{neurips_2023}

\usepackage[utf8]{inputenc} 
\usepackage[T1]{fontenc}    
\usepackage{hyperref}       
\usepackage{url}            
\usepackage{booktabs}       
\usepackage{amsfonts}       
\usepackage{nicefrac}       
\usepackage{microtype}      
\usepackage{xcolor}         

\usepackage{times}
\usepackage{amsmath, nccmath, amssymb}
\usepackage{amsthm}
\usepackage{geometry}
\usepackage{comment}
\usepackage{enumitem}
\usepackage{color}
\usepackage{graphicx}
\usepackage{bbold}
\usepackage{xfrac}
\usepackage{thmtools}
\usepackage[ruled,linesnumbered]{algorithm2e}
\usepackage{algorithmic}
\usepackage{mathtools}
\renewcommand{\leq}{\leqslant}
\renewcommand{\geq}{\geqslant}

\usepackage{notation}
\newcommand{\xs}{x_*}

\newtheorem{example}{Example} 
\newtheorem{theorem}{Theorem}
\newtheorem{lemma}{Lemma}
\newtheorem{remark}{Remark}
\newtheorem{definition}{Definition}
\newtheorem{proposition}{Proposition}
\title{SignSVRG: fixing SignSGD via variance reduction}

\author{
      Evgenii Chzhen \\
      Université Paris-Saclay, CNRS\\
      Laboratoire de mathématiques d’Orsay\\
      \texttt{first.last@cnrs.fr}
      \And
      Sholom Schechtman \\
    SAMOVAR, T\'el\'ecom Sudparis\\
       Institut Polytechnique de Paris\\
      \texttt{first.last@telecom-sudparis.eu}
      }

\begin{document}

\maketitle

\begin{abstract}
We consider the problem of unconstrained minimization of finite sums of functions.
We propose a simple, yet, practical way to incorporate variance reduction techniques into SignSGD, guaranteeing convergence that is similar to the full sign gradient descent.
The core idea is first instantiated on the problem of minimizing sums of convex and Lipschitz functions and is then extended to the smooth case via variance reduction.
Our analysis is elementary and much simpler than the typical proof for variance reduction methods.
We show that for smooth functions our method gives $\mathcal{O}(1 / \sqrt{T})$ rate for expected norm of the gradient and $\mathcal{O}(1/T)$ rate in the case of smooth convex functions, recovering convergence results of deterministic methods, while preserving computational advantages of SignSGD.
\end{abstract}

\section{Introduction}

The purpose of this work is to introduce variance reduction (VR) techniques into the analysis of sign stochastic gradient descent (\SignSGD).
Setting up the stage, we are interested in the problem of minimizing a differentiable function $f : \bbR^d \rightarrow \bbR$, which can be expressed as a finite sum of, non necessarily convex, functions:
\begin{align}\label{eq:opt_problem}
    \min_{x \in \bbR^d} \left\{f(x) \triangleq \frac{1}{n}\sum_{i = 1}^n f_i(x) \right\}\enspace.
\end{align}
In the machine learning community, this problem is referred to as empirical risk minimization, with $f_i$ corresponding to the loss function associated with a $i$-th data point.

Perhaps the most natural way to approach this problem is the celebrated stochastic gradient descent (SGD), one iteration of which is written as $x_{t+1} = x_t - \gamma g_t$, with $g_t \triangleq \nabla f_{i_t}(x_t)$ and $i_t$ is an index distributed uniformly in $\{1, \ldots, n\}$. This procedure, which goes back to \cite{RobbinsMonro51}, is quite intuitive since $\Exp[g_t]= \nabla f(x_t)$, and, on average, one iteration of SGD is pointing in the direction of the negative gradient, which tends to minimize $f$. Naturally, if one thinks that only the direction (and not the magnitude) of descent is important, then one could replace $g_t$ with its sign, leading to~\ref{eq:signSGD_intro}:
\begin{align}
\tag{SignSGD}
\label{eq:signSGD_intro}
    x_1 \in \bbR^d,\qquad x_{t+1} = x_t - \gamma \sign(g_t)\enspace .
\end{align}
To the best of our knowledge,~\ref{eq:signSGD_intro} appeared for the first time in~\cite{fabian1960stochastic} (see also~\cite[Section 6.14]{wilde1964optimum}) as a modification, in the stochastic approximation context, of the celebrated Robbins-Monro algorithm~\cite{RobbinsMonro51}. It later appeared 
(as a special case) in a sequence of works on optimal pseudo-gradient methods in~\cite{polyak1980optimal,
polyak1980robust}. In these works, (a version of)~\ref{eq:signSGD_intro} was shown to be an optimal robust iterative procedure that copes with nearly arbitrary symmetric noises. Since then, there were many studies of sign based algorithm (see e.g. \citep{riedmiller1993direct, tieleman2012lecture, KingmaBa15}).

The resurrected interest to the sign approaches is largely due to \cite{bernstein2018signsgd}, who displayed it under the new spotlight of communication efficiency that preserves many interesting properties of SGD.
A sequence of works have then emerged \cite{bernstein_sign_fault, balles2020geometry, chen2020distributed,safaryan2021stochastic,xiang2023distributed}, to name a few, all mainly working in the case of gradient plus noise oracle, similar to~\cite{polyak1980optimal,
polyak1980robust}. Furthermore, as noted in \cite{balles2018dissecting} \SignSGD might be seen as the ``limit case''of the celebrated ADAM algorithm \citep{KingmaBa15}. Thus, understanding the properties of sign based methods might shed light on the performances of ADAM and closely related algorithms.

Yet, as noted by \cite{karimireddy2019error} (and even earlier in~\cite{aved1967modification}), vanilla~\ref{eq:signSGD_intro} does not generally converge on finite sums. Intuitively, the main issue comes from the fact that, in the finite sum case, the resulting gradient noise is neither symmetric nor has zero median.
As a consequence, \cite{karimireddy2019error} proposes EF-SignSGD, which incorporates error feedback in the loop, fixing the convergence issues for convex and non-convex problems. Since then there were many different attempts to fix the convergence of~\ref{eq:signSGD_intro} \citep{chen2020distributed, safaryan2021stochastic}. Among these, the closest approach to ours is the work \cite{chen2020distributed} in which it was shown that corrupting~\ref{eq:signSGD_intro} with artificial noise might reduce its initial bias.

In a virtually unrelated series of works many authors have studied different ways to accelerate convergence of vanilla SGD. Of a particular interest to us are the variance reduction (VR) techniques which have shown that occasionally recomputing the full gradient of $f$ reduces the variance of one iteration of SGD \cite{schmidt2017minimizing, johnson2013accelerating, defazio2014saga, reddi2016stochastic}. In such a way, VR algorithms recover the convergence rates of the deterministic gradient descent while, simultaneously, preserving the computational advantages of SGD.\\
\textbf{Contributions.} In this work, we propose two simple modifications of~\ref{eq:signSGD_intro} that fixes its convergence issue. Our first method, \ref{eq:algo_core}, is based on the insight that corrupting~\ref{eq:signSGD_intro} by a uniform noise with a carefully selected amplitude fixes the convergence issues. Formally, assume for the moment, that for every $i \in \{1, \ldots, n\}$ and $x \in \bbR^d$, $\norm{\nabla f_i(x)}_{\infty} \leq G_{\infty}$ (this is for instance the case if all of the functions are $G_{\infty}$-Lipschitz w.r.t. $\ell_1$-norm). It turns out that $\Exp[\sign(g_t + G_{\infty} U)] = \nabla f(x_t) / G_{\infty}$, where $U$ is uniformly sampled in $[-1, 1]^d$. Thus, replacing $g_t$ in~\ref{eq:signSGD_intro} by $g_t + G_{\infty} U$, in average, we simply obtain an iteration of SGD, with a step-size $\gamma/ G_{\infty}$. This observation, which might have been overlooked in \cite{chen2020distributed}, immediately gives us optimal convergence rates for convex and Lipschitz continuous functions $f$.\\
In many scenarios, assuming such a uniform bound on stochastic gradients might be unrealistic. Nonetheless, it turns out that in the VR context, we actually have a control on an upper bound of stochastic iterates. Indeed, consider one iteration of the SVRG method (\cite{johnson2013accelerating, reddi2016stochastic})
$x_{t+1} = x_{t} - \gamma v_t$, with $v_t = \nabla f_{i_t}(x_t) - \nabla f_{i_t}(\tilde x) + \nabla f(\tilde x)$ and where $\tilde x$ is one of the previous iterates and is called reference point. If, $\nabla f_i$ is $L$-Lipschitz continuous, then $\norm{v_t} \leq L \norm{x_t - \tilde{x}} + \norm{\nabla f(\tilde{x})}$. Thus, as soon as $\norm{x_t - \tilde{x}}$ is not too large (say smaller than $D >0$), we have a bound on the norm of $v_t$, which involves the value of $\nabla f$ at the reference point $\tilde{x}$.

This observation, motivates the introduction of our main method \emph{sign stochastic variance reduction gradient} \eqref{eq:SignSVRG} (see Algorithm~\ref{algo:SignSVRG} for its two variants), one iteration of which is given by:
\begin{equation}
    \label{eq:SignSVRG}
    \tag{SignSVRG}
    x_{t+1} = x_t - \gamma \sign(\nabla f_{i_t}(x_t) - \nabla f_{i_t}(\tilde x) + \nabla f(\tilde x) + G_t U_t)\enspace ,
\end{equation}
where $U_t$ is uniformly sampled in $[-1, 1]^d$, and $G_t \triangleq L D + \norm{\nabla f(\tilde x) }$ (or coordinate-wise noise for \Vartwo in Algorithm~\ref{algo:SignSVRG}). In such a way, on average \ref{eq:SignSVRG} goes to the direction of the negative gradient of $f$. We establish that, under an appropriate choice of $D$, \ref{eq:SignSVRG} recovers the convergence rates of the deterministic sign gradient descent for (convex and non-convex) smooth functions $f$. Furthermore, similarly to the initial SVRG method we show that the update of {reference point} $\tilde{x}$ happens rarely enough to preserve the computational advantages of stochastic methods.\\
Finally, we believe that our proof techniques are fairly general, opening a way of incorporating VR and sign techniques into other methods than SGD.\\
\textbf{Notations.} Euclidean norm in $\bbR^d$ is denoted by $\|\cdot\|$, while $\ell_p$-norm by $\|\cdot\|_p$. For every linear operator $A : (\bbR^d, \|\cdot\|_p) \to (\bbR^d, \|\cdot\|_q)$ with $p, q \geq 1$ being H\"older conjugates, we denote by $\|A\|_{p \mapsto q}$ its operator norm. Whenever $p = q = 2$, we write $\|A\|$ to denote the spectral norm of $A$. For $u \in \bbR^d$, we write $u^i$, to denote $i$-th coordinate of $u$. For $u, v \in \bbR^d$, we denote $u \otimes v \in \bbR^d$ the vector which $i$-th coordinate is equal to $u^i v^i$. For an abstract set $A$ we denote by $\unif(A)$, the uniform distribution on $A$, whenever the sigma algebra is clear. We denote by $\sign(a)$, the coordinate-wise sign operator. We denote by $x_*$ a global minimizer of $f$ over $\bbR^d$.

\section{The core trick: convergence rates for finite sums}\label{sec:trick}
The goal of this section is to showcase, on a basic example, the simple trick that fixes the non-convergence issue of vanilla~\ref{eq:signSGD_intro} for finite sums. In particular, in this section, we do not intend to push the limits of this trick and obtain neither optimal dependencies nor adaptive versions of our algorithm.
Through this section we work under the following assumption.
\begin{assumption}
    \label{ass:convex+lip}
    Each $f_i : \bbR^d \to \bbR$ is convex and $G_{\infty}$-Lipschitz w.r.t. $\ell_1$-norm.
\end{assumption}
Assuming that each function is $G_{2}$-Lipschitz w.r.t., the Euclidean norm, stochastic gradient descent with step size $\|x_1 - x_*\|/{G_{2}\sqrt{T}}$ yields $O(1 / \sqrt{T})$ as a convergence rate. Let us show how similar rate can be achieved using a mild modification of~\ref{eq:signSGD_intro} under a slightly different smoothness assumption.
Consider the following algorithm:
\begin{align}
    \tag{SignSGD+}
    \label{eq:algo_core}
    x_{t+1} = x_t - \gamma \sign(g_t + G_{\infty} U_t)\enspace,
\end{align}
where $g_t$ is a subgradient of $f_{i_t}$ at $x_t$ with $i_t \sim \unif(\{1, \ldots, n\})$, $U_t \sim \unif([-1, 1]^d)$ and independent from $x_t, i_t$. The only difference of~\ref{eq:algo_core} with the vanilla~\ref{eq:signSGD_intro} is the additional corruption by uniform noise with carefully chosen magnitude under the $\sign$ operator. Actually, such a simple trick fixes the non-convergence issue as stated in the next result.
\begin{theorem}
    Let Assumption~\ref{ass:convex+lip} be satisfied. Let $x_1 \in \bbR^d$ and $\gamma = \tfrac{\|x_1 - x_*\|}{\sqrt{dT}}$, then~\ref{eq:algo_core} satisfies
    \begin{align*}
    f(\bar{x}_T) - f(x_*) \leq G_{\infty}\norm{x_1 - x_*}\sqrt{\frac{d}{T}}\quad\text{where } \bar{x}_T = \sum_{t \leq T} x_t\enspace.
    \end{align*}
\end{theorem}
\begin{proof}
Using the definition of $x_{t+1}$ in~\ref{eq:algo_core}, we obtain
\begin{align*}
    \Exp[\|x_{t+1} - x_*\|^2 \mid x_t, i_t]  = \|x_t - x_*\|^2 + 2\gamma \sum_{j = 1}^d (x_t^j - x_*^j)\Exp[\sign(g_t^j + G_{\infty} U_t^j) \mid x_t, i_t] + \gamma^2 d\enspace.
\end{align*}
Since $|g_t^j| \leq G_{\infty}$ almost surely for all $j, t$, it holds that

\begin{align}
    \label{eq:key}
    \fbox{$\Exp[\sign(g_t^j + G_{\infty} U_t^j) \mid x_t, i_t] = \frac{g_t^j}{G_{\infty}}$}\enspace.
\end{align}
Combining the above with the penultimate equality and taking total expectation, we deduce that
\begin{align}
    \label{eq:sign_to_no_sign}
    \Exp[\|x_{t+1} - x_*\|^2]  = \Exp[\|x_t - x_*\|^2] -\frac{2\gamma}{G_{\infty}} \Exp[\scalar{x_t - x_*}{\partial f(x_t)}] + \gamma^2 d\enspace.
\end{align}
Convexity of $f$, implies that
\begin{align*}
    \Exp[f(x_t) - f(x_*)] \leq G_{\infty}\parent{\frac{\Exp[\|x_t - x_*\|^2] - \Exp[\|x_{t+1} - x_*\|^2]}{2\gamma} + \frac{\gamma d}{2}}\enspace.
\end{align*}
Summing up the above bounds for $t = 1, \ldots, T$ gives a bound on the regret
\begin{align*}
    \sum_{t = 1}^T\Exp[f(x_t) - f(x_*)] \leq \frac{G_{\infty}}{2}\parent{\frac{\|x_1 - x_*\|^2}{\gamma} + \gamma d T}\enspace.
\end{align*}
The result follows from the choice of $\gamma$ and online-to-batch conversion using convexity of $f$.
\end{proof}
One can notice that the same proof works in the setup of online convex optimization. Subjectively, the proof for this setup is even simpler than for the finite sums case (since there is less randomness involved).
One also observes that the proof passes with minimal modifications to the case of projected gradient descent, mirror descent, and their dual averaging alternatives.
\begin{remark}[Clear limitation]
    We note that unlike, for example, vanilla gradient descent, where the adaptation to $G_{\infty}$ can be achieved almost for free with a proper version of AdaGrad~\cite{duchi2011adaptive, streeter2010less}, extending these ideas to the above algorithm seems non-trivial and is left for future works.
\end{remark}
\textbf{Lessons learnt.} While simple, this construction gives one key insight---for Eq.~\eqref{eq:key} to hold, we can use \emph{any} upper bound on the infinity norm of $g_t$ (or its $j$th component). Crucially, this upper bound should not depend on the randomness of $i_t$, yet it could depend on the point $x_t$ (i.e., picking $\|\nabla f_{i_t}(x_t)\|_\infty$ does not work, while $\max_{i = 1, \ldots, n}\|\nabla f_i(x_t)\|_\infty$ does). Furthermore, Eq.~\eqref{eq:sign_to_no_sign} suggests that our procedure behaves (on average) as the usual SGD iterated in the ``scalar product part'' while preserving the ``constant speed'' feature in the third quadratic term. Precisely these insights lead us to consider a variance reduction scheme to fix~\ref{eq:signSGD_intro}.

\section{Smoothness assumption for variance reduction}
Before presenting our variance reduction approach for~\ref{eq:signSGD_intro}, let us start by introducing and discussing main assumptions used in this work.
\begin{assumption}
\label{ass:smooth}
There is $p,q\geq1$, H\"older conjugates, and $L_q >0$ such that for all $i \in \{1, \ldots, n\}$, $f_i$ is continuously differentiable and for all $x, y \in \bbR^d$,
\begin{align*}
    \|\nabla f_i(x) - \nabla f_i(y)\|_{p} \leq L_{q} \|x - y\|_q\enspace.
\end{align*}
\end{assumption}
Assumption~\ref{ass:smooth} allows us to control the behavior of $f$ and $\nabla f$, which is shown by the following standard lemmas. For completeness, the proofs are provided in Appendix~\ref{app:proofs}.
\begin{lemma}\label{lm:descent}
Under Assumption~\ref{ass:smooth}, for $x,y \in \bbR^d$, the following holds
\begin{equation*}
    f(y) \leq f(x) + \scalar{\nabla f(x)}{y-x} + \frac{L_q}{2}\norm{y-x}_q^2 \enspace .
\end{equation*}
\end{lemma}

\begin{lemma}\label{lm:gap_lbound}
    Let Assumption~\ref{ass:smooth} hold and assume that $f$ is convex and is lower bounded by $f_*$. For all $x \in \bbR^d$, the following holds
    \begin{equation*}
        {\norm{\nabla f(x)}_{p}^2} \leq 2L_q(f(x)-f_*) \, .
    \end{equation*}
\end{lemma}
Apart the insights gained in Section~\ref{sec:trick}, these are essentially the only two results that we will need to establish theoretical guarantees for our method.

\paragraph{On the dependency of Assumption~\ref{ass:smooth} on the choice of norms.}
In most recent works on sign based algorithms, Assumption~\ref{ass:smooth} was mainly deployed with $p = 1$ and $q = +\infty$~\cite{balles2020geometry,bernstein2018signsgd} (i.e., they use $L_{\infty}$ to exhibit the rates of convergence). While we provide analysis for any conjugate pair of $p, q \geq 1$, we argue in this section that our eventual bounds are the most advantageous with $p = \infty$ and $q = 1$.
First of all, inspecting the values of $L_q$'s, we find out that the smallest one corresponds to $q = 1$. It is rather intuitive---on the left we have the smallest $\ell_{\infty}$ and on the right the largest $\ell_1$ norms.
\begin{lemma}
    \label{lem:Lq_order}
   For all $q \geq 1$ we have $L_{1} \leq L_{q}$.
\end{lemma}
\begin{proof}
Assume for simplicity that $f_i$ is twice continuously differentiable, then
\begin{align*}
    L_1 = \|\nabla^2 f_i(x)\|_{1 \to \infty}
    &\triangleq \max\enscond{\|\nabla^2 f_i(x) u\|_{\infty}}{\|u\|_1 \leq 1}\\
    &\leq \max\enscond{\|\nabla^2 f_i(x) u\|_{p}}{\|u\|_q \leq 1} \triangleq \|\nabla^2 f_i(x)\|_{q \to p}\enspace,
\end{align*}
for all $p, q \geq 1$ and $x \in \bbR^d$. In particular, for $p, q \geq 1$ such that $1/q + 1/p = 1$.  
\end{proof}
Furthermore, let us give a simple example of a function, where $L_1$ depends on the dimension in a favourable way, that is, it is of order $1/ d$, while $L_{2} = 1$, and $L_{\infty}$ explodes with the dimension.
\begin{example}
    \label{ex:smoothness}
    Assume that $f(x) = \tfrac{1}{2}(\zeta^{\top} x)^2$, where $\zeta$ is distributed uniformly on the unit sphere, then
    \begin{align*}
        \Exp[L_1] \leq \frac{c\log(ed)}{d},\qquad L_2 = 1, \qquad \Exp[L_{\infty}] \geq \frac{d}{2} \enspace,
    \end{align*}
    for some universal $c > 0$ and almost surely for the second equality.
\end{example}
\begin{proof}
We have $\nabla^2 f(x) = \zeta\zeta^\top$. We are going to use the fact that there exists standard Gaussian $W$ such that $\zeta = W / \|W\|$ and that $W$ is independent of $W / \|W\|$~\cite[see e.g.,][Lemma 1]{SchechtmanZinn00}. We have
\begin{align*}
    \Exp[L_1] = \Exp\|\zeta\zeta^{\top}\|_{1 \to \infty} = \Exp\left[\frac{\max_{i, j}|W_iW_j|}{\|W\|^2}\right] = \frac{1}{d}\Exp\left[\max_{i, j}|W_iW_j|\right]\lesssim \frac{\log(ed)}{d}\enspace,
\end{align*}
where last inequality was obtained from the known result on the expected maxima of chi-square random variables~\cite[see e.g.,][Example 2.7]{boucheron2013concentration} using the fact that $4W_iW_j = (W_i - W_j)^2 - (W_i + W_j)^2$. Meanwhile, clearly
\begin{align*}
    L_2 = \|\zeta\zeta^{\top}\|_{2 \to 2} = 1 \qquad\text{almost surely}\enspace.
\end{align*}
Finally, for the $L_{\infty}$, it holds that
\begin{equation*}
    \Exp[L_{\infty}] = \Exp\norm{\zeta\zeta^\top}_{\infty \to 1} = \Exp\|\zeta\|_1^2 = \frac{\Exp[\|W\|_1^2]}{d} = \frac{d + {\frac{2}{\pi}}(d^2 - d)}{d} = \parent{1 - \frac{2}{\pi}} + \frac{2}{\pi}d\enspace.\qedhere
\end{equation*}
\end{proof}
Thus, for this example, $L_1 \asymp_{\log d} L_2 / d$, while $L_2$ is dimension independent. Our bounds will eventually involve expressions of the form $d L_1$, the above discussion suggests that these quantities can morally be seen as constants. Let us also remark that all our proof techniques work with $L_{\infty}$, making the dimension dependency virtually disappear from some bounds. Yet, we believe that such a view is faulty since $L_{\infty}$ does depend on the dimension in many situations as highlighted in the above example. Let us finally remark that one can obtain high probability bounds instead of in-expectation relying on, for example,~\cite[Lemma 1 of][]{laurent2000adaptive}.

\section{Main results: SignSVRG \Varone and \Vartwo}

\begin{algorithm}[t!]
\DontPrintSemicolon 
\caption{SignSVRG (\Varone) and (\Vartwo)}\label{algo:SignSVRG}
\SetKwInput{Input}{Input}
   \SetKwInOut{Output}{Output}
   \SetKwInput{Initialization}{Initialization}
   \Input{$x_1, L_q, D, \gamma$,}
   \Initialization{$k(1) = 1$ ($k(t)$ denotes the number of reference points up to time $t$)}

$\tilde{x}_{k(1)} \gets x_1$ \tcp*[r]{{\small set initial reference point}}

$g_{k(1)} \gets \nabla f(x_1)$ \tcp*[r]{{\small compute full gradient at the beginning}}

\For{$t=1,\ldots,T$}{
    $i_t \sim \text{Unif}(\{1, \ldots, n\})$ \tcp*[r]{{\small sample function}}

    $U_t \sim \text{Unif}([-1, 1]^d)$ \tcp*[r]{{\small sample corruption}}

   \If(\tcp*[f]{\small for \Varone}){\Varone}{
   $G_t^1 \gets L_q \|x_t - \tilde{x}_{k(t)}\|_q + \|g_{k(t)}\|_{p}$ \tcp*[r]{{\small compute upper bound}}
   
   $G_t = (G_1^1, \ldots, G_t^1)^{\top}$
   }

   \If(\tcp*[f]{\small for \Vartwo}){\Vartwo}{
$G_t^j \gets L_q \|x_t - \tilde{x}_{k(t)}\|_q + |g_{k(t)}^j|$ \tcp*[r]{{\small compute upper bound}}

    $G_t = (G_t^1, \ldots, G_t^d)^{\top}$
   
   }

    $x_t^+ \gets x_t - \gamma \sign(\nabla f_{i_t}(x_t) - \nabla f_{i_t}(\tilde{x}_{k(t)}) + g_{k(t)} + G_t \otimes U_t)$ \tcp*[r]{{\small next candidate}}

    \If(\tcp*[f]{\small if not too far from the reference point}){$\|x_{t}^+ - \tilde x_{k(t)}\|_q \leq D$\label{line:check}}{
    $x_{t + 1} \gets x_t$ \tcp*[r]{{\small make a step}}
    
    $k(t+1) \gets k(t)$ \tcp*[r]{{\small keep the same reference point}}
    
    }
    \Else(\tcp*[f]{\small otherwise, update reference point}){
    $k(t+1) \gets k(t) + 1$ \;
    $\tilde{x}_{k(t + 1)} \gets x_t$ \;
    $g_{k(t + 1)} \gets \nabla f(x_t)$ \tcp*[r]{{\small recompute full gradient}}
    
    }
    }

\end{algorithm}

In this section we present our main algorithm SignSVRG, which combines ideas from Section~\ref{sec:trick} with a variance reduction technique. It is formally presented in Algorithm~\ref{algo:SignSVRG}.
\paragraph{Variance reduction methods.} The idea behind variance reduction methods is to combine computational advantages of SGD with good convergence properties of deterministic gradient descent. This is done by recomputing, from time to time, the full gradient of $f$, and using this estimation to reduce the variance of stochastic iterates. For instance, one iteration of the celebrated SVRG~\cite{johnson2013accelerating, reddi2016stochastic} is written as $x_{t+1} = x_{t} - \gamma v_t$, with $v_t = \nabla f_{i_t}(x_t) - \nabla f_{i_t}(\tilde x) + \nabla f(\tilde x)$ and where $\tilde x$ is one of the previous iterates. Note that $v_t$ is still an unbiased estimator of $\nabla f(x_t)$: $\Exp[v_t] = \nabla f(x_t)$. Moreover, $\Var(v_t) = \Exp \norm{\nabla f_{i_t}(x_t) - \nabla f_{i_t}(\tilde x) + \nabla f(x_t)- \nabla f(\tilde x)}^2 \leq 2 L_2^2 \Exp\norm{x_t - \tilde x}^2$. As $x_t$ and $\tilde x$ are expected to converge to a critical point $x_*$, the variance of $v_t$ goes to zero, recovering convergence rates of the deterministic gradient descent. Besides, such rates are obtained without updating the reference point $\tilde{x}$ too often---computational requirements of SVRG are comparable to that of SGD.
\paragraph{SignSVRG (\Varone) description.} We are now ready to present our method: \emph{SignSVRG}, which generates its iterates according to Algorithm~\ref{algo:SignSVRG}. The starting assumption of SignSVRG is the same as in usual variance reduction methods: to accelerate the convergence we are willing, from time to time, to compute the full gradient of our objective. The new insight is that, in this way, we actually have a control on the upper bound of our stochastic gradients. Indeed, notice that in one iteration of SVRG, $\norm{v_t}_p \leq \norm{\nabla f_{i_t}(x_t) - \nabla f_{i_t}(\tilde x)}_p + \norm{\nabla f(\tilde x)}_p \leq L_q \norm{x_t - \tilde x}_q + \norm{\nabla f(\tilde x)}_p \triangleq G_t $. Hence, writing an iteration of the form $x_{t+1} = x_t - \gamma \sign (v_t + G_t U_t)$, similarly to Section~\ref{sec:trick}, on average, we obtain an algorithm that goes in the direction of the negative gradient.
A key difference of our method with SVRG lies in the fact that $\Var(\sign(v_t + G_t U_t)) \approx d$ is not expected to go to zero. For this reason, we are not expecting to have the usual properties of variance reduction methods such as an average decrease after one epoch or a linear convergence for strongly convex functions~\cite[see e.g.,][]{reddi2016stochastic,pmlr-v108-gorbunov20a}. Nevertheless, we are able to show that our algorithm, on smooth convex and non-convex functions, recovers rates of the deterministic gradient descent. We present two variants of SignSVRG---\Varone enjoys non-convex and convex rates of convergence, while \Vartwo closely mimics~\ref{eq:def_signGD}.

\begin{theorem}\label{thm:main_non_cvx}
    Let Assumption~\ref{ass:smooth} be satisfied, then for all $D, \gamma >0$, \Varone of Algorithm~\ref{algo:SignSVRG} satisfies
        \begin{align*}
    \frac{1}{T}\Exp \left[\sum_{t=1}^{T} \frac{\norm{\nabla f(x_t)}^2}{2L_{q} D  + \norm{\nabla f(x_t)}_{p}}  \right]\leq \frac{\Exp[f(x_1)  - f(x_{T+1})]}{T\gamma}  + \frac{L_{q}\gamma d^{\sfrac{2}{q}}}{2} \enspace ,
\end{align*}
and \Vartwo of Algorithm~\ref{algo:SignSVRG} satisfies
  \begin{align*}
    \frac{1}{T} \Exp\left[ \sum_{t=1}^T\frac{\norm{\nabla f(x_t)}_1^2}{ 2L_q d D + \norm{\nabla f(x_t)}_1}\right] \leq \frac{\Exp[f(x_1)  - f(x_{T+1})]}{T \gamma } + \frac{L_q \gamma d^{\sfrac{2}{q}}}{2} \enspace .
\end{align*}
\end{theorem}
\begin{proof}
    Fix some $t = 1, \ldots, T - 1$ so that $k({t+1}) = k$. By Lemma~\ref{lm:descent} we have
\begin{align*}
    f(x_{t+1}) \leq f(x_t) - \gamma\scalar{\nabla f(x_t)}{\sign\parent{\nabla f_{i_t}(x_t) - \nabla f(\tilde{x}_{k(t)}) + \nabla f(\tilde{x}_{k(t)}) + G_t \otimes U_t}} + \frac{L_{q}}{2}\gamma^2 d^{\sfrac{2}{q}}\enspace.
\end{align*}
Note that for both variants, the $j$-th coordinate of $\nabla f_{i_t}(x_t) - \nabla f_{i_t}(\tilde{x}_{k(t)}) + \nabla f(\tilde{x}_{k(t)})$ is almost surely bounded by $G_t^j$. Thus, taking the expectation w.r.t. $H_t = \sigma(x_1, \ldots, x_t, U_1, \ldots, U_{t-1})$, we deduce that almost surely
\begin{align}
    \label{eq:non_cvx}
    \Exp[f(x_{t+1}) \mid H_t] \leq f(x_t) - \gamma \sum_{j=1}^d\frac{(\nabla_j f(x_t))^2}{G_t^j} + \frac{L_{q}}{2}\gamma^2 d^{\sfrac{2}{q}}\enspace.
\end{align}
By design (see line~\ref{line:check} of Algorithm~\ref{algo:SignSVRG}) and thanks to the Lipschitzness of the gradient, for $j \in \{1, \dots, d\}$, we have almost surely
\begin{align*}
 \textrm{\Varone:}\quad   G_t^j \leq 2L_{q} D  + \norm{\nabla f(x_t)}_{p} \qquad \textrm{or} \qquad \textrm{\Vartwo:} \quad G_t^j \leq 2 L_q D + |\nabla_j f(x_t)|\enspace.
\end{align*}
Thus, taking total expectation in~\eqref{eq:non_cvx} and re-arranging, we obtain, for the first variant,
\begin{align}\label{eq:svrg_desc1}
    \Exp \left[ \frac{\norm{\nabla f(x_t)}^2}{2L_{q} D  + \norm{\nabla f(x_t)}_{p}} \right] \leq \frac{\Exp[f(x_t)  - f(x_{t+1})]}{\gamma}  + \frac{L_{q}}{2}\gamma d^{\sfrac{2}{q}}\enspace,
\end{align}
and for the second variant (\Vartwo), 
\begin{align}\label{eq:svrg_desc2}
  \Exp\left[ \frac{\frac{1}{d}\norm{\nabla f(x_t)}_1^2}{ 2L_q D + \frac{1}{d}\norm{\nabla f(x_t)}_1}\right] \leq\Exp\left[ \sum_{j = 1}^d\frac{(\nabla_j f(x_t))^2}{2L_qD + |\nabla_j f(x_t)|} \right] \leq \frac{\Exp[f(x_t) {-} f(x_{t+1})]}{\gamma} {+} \frac{L_{q}}{2}\gamma d^{\sfrac{2}{q}}\,, 
 \end{align}
where we have used  the fact that for any $b >0$, $x \mapsto x^2/(x+b)$ is convex for $x \geq 0$. Summing up Eq.~\eqref{eq:svrg_desc1} and~\eqref{eq:svrg_desc2} concludes the proof.
\end{proof}

An appropriate choice of the step-size $\gamma > 0$, immediately gives the following convergence rate for both variants of SignSVRG on smooth non-convex objectives.
\begin{corollary}
\label{cor:non_cvx}
Let Assumption~\ref{ass:smooth} be satisfied and assume that $f$ is bounded below by $f_*$. Fix $T>0$, $D = \sfP\sqrt{2 / L_q T}$ for some $\sfP > 0$ and $\gamma = \tfrac{1}{d^{\sfrac{1}{q}}}\sqrt{\tfrac{2}{L_q T}}$. Then, \Varone of Algorithm~\ref{algo:SignSVRG} satisfies
\begin{align*}
     \Exp[\norm{\nabla f(\bar{x}_T)}_{p}] \leq 2\sfP\sqrt{\frac{2L_q}{T}}\qquad\text{or}\qquad \frac{\Exp[\norm{\nabla f(\bar{x}_T)}]^2}{\Exp[\norm{\nabla f(\bar{x}_T)}_{p}]} \leq d^{\sfrac{1}{q}} (f(x_1)  - f(x_*)  + 1)\sqrt{\frac{2L_q}{T}}\enspace,
 \end{align*}
where $\bar{x}_T$ is a random point along the trajectory of the algorithm. Similarly, \Vartwo of Algorithm~\ref{algo:SignSVRG} satisfies
\begin{equation*}
    \Exp[\norm{\nabla f(\bar x_T)}_1] \leq \sqrt{\frac{2L_q}{T}}\max \left( d^{\sfrac{1}{q}} (f(x_1)  - f(x_*)  + 1),\,  2d\sfP\right) \enspace .
\end{equation*}
\end{corollary}
In particular, for \Varone (similar for \Vartwo, but with different norm), to sample a point $x$ with
\begin{equation}\label{eq:sam_point}
    \max \left(   \Exp[\norm{\nabla f(x)}_{p}]\, , \, \frac{\Exp[\norm{\nabla f(x)}]^2}{\Exp[\norm{\nabla f(x)}_{p}]}\right) = \cO(\varepsilon) \enspace,
\end{equation}
we need to run Algorithm~\ref{algo:SignSVRG} for $T = \cO(\varepsilon^{-2})$ iterations (note that the above bound is the most advantageous with $p = \infty$ and $q = 1$ since $\|\cdot\|_{\infty} \leq \|\cdot\|_2$). Ignoring, for the moment, the dependency on the dimension, this convergence rate matches the one of the deterministic gradient descent as is achieved by SVRG~\cite{johnson2013accelerating,reddi2016stochastic}.

It is also informative to compare the rate of \Vartwo with that of the full~\ref{eq:def_signGD}, that takes the form
\begin{align}
    \tag{SignGD}
    \label{eq:def_signGD}
    x_{t+1} = x_t - \gamma \sign(\nabla f(x_t))\enspace,
\end{align}
We prove in appendix that choosing $\gamma =\frac{1}{d^{1/q}} \sqrt{\frac{2}{TL_q}}$ (same step-size),~\ref{eq:def_signGD} satisfies
\begin{equation*}
    \Exp\norm{\nabla f(\bar{x}_T)}_1 \leq d^{1/q}\sqrt{\frac{L_q}{2T}}(f(x_1) - f_* + 1) \enspace.
\end{equation*}
That is, \Vartwo has nearly the same convergence guarantee as the full~\ref{eq:def_signGD}, while preserving some computational advantages of the~\ref{eq:signSGD_intro}. If only the convergence w.r.t. the number of iterations $T$ is of concern, same conclusion is applicable to the first variant (\Varone) of our algorithm.

Let us discuss the dimension dependency of the second variant (\Vartwo) and argue that in some favourable cases the guarantee of Corollary~\ref{cor:non_cvx} can be seen as dimension-free.
\begin{remark}[On the dimension]
    \label{rem:dimensions}
    Corollary~\ref{cor:non_cvx} involves two rates for both variants.
    Let us mainly focus on the second variant (\Vartwo) for this discussion as the same arguments apply for the first one.
    For $q=\infty$, the commonly used setup in the literature, our bound reads as
    \begin{equation*}
        \Exp[\norm{\nabla f(\bar x_T)}_1] \leq \sqrt{\frac{2L_{\infty}}{T}}\max \left( (f(x_1)  - f(x_*)  + 1),\,  2d\sfP \right) \enspace .
    \end{equation*}
    Thus, in the unfavourable case of Example~\ref{ex:smoothness}, the right hand-side of the above bound growth as $d^{3/2}$. In other words, even if $\norm{\nabla f(\bar x_T)}_1 \approx \sqrt{d}\norm{\nabla f(\bar x_T)}_2$ (which happens on non-sparse gradients) the dependency on the dimension is still present. In contrast, with $q = 1$ and assuming the favourable case of Example~\ref{ex:smoothness}, $L_1 \approx 1/d$ we get 
    \begin{align*}
        {\Exp[\norm{\nabla f(\bar{x}_T)}_1]} \lesssim \sqrt{\frac{d}{T}}\max \left( (f(x_1)  - f(x_*)  + 1)/d,\,  \sfP \right)\enspace.
    \end{align*}
    In particular, for the usual case of $\norm{\nabla f(\bar{x}_T)}_1 \approx \sqrt{d}\norm{\nabla f(\bar{x}_T)}_2$, our bound could actually be dimension independent for the Euclidean norm of the gradient, which is not the case for any other geometry. Thus, we believe that the case of $q=1$ is intrinsic for SignSVRG. 
 \end{remark}
\paragraph{Frequency of reference point update.}
In view of Remark~\ref{rem:dimensions}, we fix $q = 1$.
Naturally, one is interested by the question of the frequency of the updates of the reference point $\tilde{x}$. Indeed, if we update it too often, then the computational complexity will be similar to the one of the deterministic (sign) gradient descent, making it prohibitive for large scale finite sums. We first note that in Algorithm~\ref{algo:SignSVRG}, $k(T)$ stands for the total number of reference point updates. The reference point is updated iff the condition in line~\ref{line:check} of Algorithm~\ref{algo:SignSVRG} is met (for both variants). We can state the following result.
\begin{lemma}
    \label{lem:ref_point_freq}
    Let assumptions of Corollary~\ref{cor:non_cvx} be satisfied. The number of reference point updates $k(T)$, for both variants of Algorithm~\ref{algo:SignSVRG}, satisfies $k(T) \leq \lceil \sfrac{T}\sfP \rceil$.
\end{lemma}
\begin{proof}
Denote by $\tau$ the first iteration $t$ such that $k(t) = k({\tau})$ in Algorithm~\ref{algo:SignSVRG}---the first moment of reference point update for $t$. For $t \leq \sfP$, we obtain by the triangle inequality that
\begin{align*}
    \norm{\tilde{x}_{k(\tau)} - x_{\tau + t}}_1 \leq t\gamma d \leq {\sfP}\sqrt{2 / {L_1T}}\enspace,
\end{align*}
where the last inequality follows by substitution of $\gamma$ from Corollary~\ref{cor:non_cvx} and our assumption $t \leq \sfP$. In other words, once the reference point is updated, our algorithm after $\sfP$ iterations stays withing $\tfrac{\sfP\sqrt{2}}{\sqrt{L_1T}}$ radius of the reference point in $\ell_1$-norm. Setting $D = \tfrac{\sfP\sqrt{2}}{\sqrt{L_1T}}$, as in Corollary~\ref{cor:non_cvx}, we are sure that the reference point is updated at most every $\sfP$ rounds, that is $k(T) \leq \lceil \tfrac{T}\sfP \rceil$ (i.e. $\sfP$ stands for period).
\end{proof}
Note that Lemma~\ref{lem:ref_point_freq} only gives a rough upper bound and the actual update frequency can be much lower.
\paragraph{Communication overhead.} The choice of period $\sfP$ in Corollary~\ref{cor:non_cvx} also affects the communication efficiency of the algorithm. For example, for $\sfP = n$ the algorithm communicates the full gradient once in every $n$ iterations and the other $n - 1$ iterations it communicates $d$-bits. That is, for $P = n$, the total number of bit communications after $T$ iterations (assuming that float takes $32$ bits) is at most $dT(32 + (1 - 1/n))$. For example, vanilla~\ref{eq:signSGD_intro} (which does not converge) would communicate $dT$ bits after $T$ rounds and vanilla SGD $32dT$. That is, the gain in communication efficiency is marginal for $\sfP = n$.
In contrast, setting $\sfP = \tfrac{32n - 1}{\beta}$ for some $\beta > 0$, reduces the number of communicated bits after $T$ rounds to at most $(1 + \beta) dT$, which nearly matches that of~\ref{eq:signSGD_intro} and introduces a factor of $O(n)$ in the convergence rate. We state the following result.
\begin{lemma}
    Let assumptions of Corollary~\ref{cor:non_cvx} be satisfied and $\sfP \in \mathbb{N}$. Assuming that float takes $\mathsf{F}$ bits, gradient communications to the center for both variants of Algorithm~\ref{algo:SignSVRG} after $T$ iterations is at most $d(\mathsf{F}n + \sfP - 1) \lceil\tfrac{T}{\sfP} \rceil$ bits.
\end{lemma}
Thus, for practical purposes it seems that $\sfP = \mathsf{F}n$ is a reasonable choice, where $\mathsf{F}$ is the size of float.

\subsection{Smooth convex case}
In the previous section, we have established that SignSVRG is recovering convergence rates of the deterministic sign gradient descent.
Note that there are no known convergence rates for~\ref{eq:def_signGD} on convex and Lipschitz function (without strong convexity). A notable exception is the work of~\cite{moulay2019properties}, where authors require additional, rather unnatural, assumption to establish convergence of~\ref{eq:def_signGD} (see their Theorem 1 and Eqs.~(4), (5)).
Yet, quite surprisingly, the first variant (\Varone) of our method also enjoys improved rates when $f$ is supposed to be convex. We emphasize that this is not the case for \Vartwo, which follows~\ref{eq:def_signGD} too closely. We work here under the following assumption.
\begin{assumption}
    \label{ass:convex}
    The functions $f$ is convex and admits a global minimizer.
\end{assumption}
Denote $f_* = \min f(x)$ and $\xs$ such that $f(x_*) = f_*$. The main result of this section is the convergence rate of the first variant (\Varone) of SignSVRG on convex smooth functions. As before, we first state the result that is valid for any fixed step-size $\gamma$.
\begin{theorem}
    \label{thm:main_cvx}
Let Assumptions~\ref{ass:smooth} and \ref{ass:convex} hold. For all $D, \gamma > 0$, \Varone of Algorithm~\ref{algo:SignSVRG} satisfies
    \begin{align*}
     \frac{1}{T} \sum_{t=1}^{T} \Exp \left[\frac{ f(x_t)- f_{*}}{2L_q D + \sqrt{ 2 L_q (f(x_t)- f_{*})}} \right] \leq \frac{\norm{x_1 - \xs}^2}{2T \gamma} + \frac{\gamma d}{2}\enspace .
\end{align*}
\end{theorem}
As in the previous section, the convergence rate is obtained by carefully setting the step-size $\gamma > 0$.
\begin{corollary}\label{cor:cvx}
Let Assumption~\ref{ass:smooth} and \ref{ass:convex} be satisfied. \Varone of Algorithm~\ref{algo:SignSVRG} with parameters $\gamma = \alpha /(\sqrt{d T})$ and $D = \sfP/\sqrt{T}$, for some $\sfP>0, \alpha > 0$ satisfies
\begin{align*}
\Exp[ \sqrt{f(\bar x_T) - f_*}] \leq \sqrt{\frac{2L_q}{T}}\max\left( \sfP\, ,  \sqrt{d}  \left( \frac{\norm{x_1 - \xs}^2}{\alpha} + \alpha\right) \right)\enspace.
 \end{align*}
where $\bar{x}_T$ is a random point along the trajectory of the algorithm.
\end{corollary}
The above result perfectly highlights our decision to insist on $q = 1$. Indeed, due to Lemma~\ref{lem:Lq_order}, the smallest $L_q$ corresponds to $q = 1$ and there are non-trivial situations when $L_1 \approx 1 / d$, making the above convergence rate actually dimension independent. To the best of our knowledge, the above result is the first convergence rate for convex smooth functions of stochastic sign based methods. We remark, however, that the bound is obtained on the expected root of the optimization error which is always worse than the root of expectation by Jensen's inequality. 






\subsection{On the absence of linear rates for strongly convex functions}

A reader acquainted with VR techniques might be tempted to ask if our method enjoys improved rates in the case where every $f_i$ is a strongly convex function. Indeed, it is known that in this case deterministic gradient descent converges linearly and variance reduction methods~\cite{defazio2014saga, reddi2016stochastic, johnson2013accelerating,pmlr-v108-gorbunov20a} are able to recover such rates. In this section, we argue that this is not achievable for SignSVRG. To this end, we give an example of a strongly convex function, where~\ref{eq:def_signGD} does not converge linearly for any fixed step-size $\gamma > 0$.

\begin{example}[\ref{eq:def_signGD} does not converge linearly with constant step-size]
    Consider the iteration of the form
    \begin{align*}
        x_{t+1} = x_t - \gamma \sign(\nabla f(x_t))\enspace.
    \end{align*}
    Let $f(x) = \tfrac{1}{2}x^2$ and note that $\sign(\nabla f(x)) = \sign(x)$. Already at this level, we can see that expecting linear rates is unreasonable---\ref{eq:def_signGD} on $f$ is simply sub-gradient descent on $|x|$ which is non-smooth and non-strongly convex. Yet, we can make this intuition more formal. Noticing that the global minimum of $f$ is $f_* = 0$,
    setting $\delta_{t} = \sqrt{f(x_{t+1}) - f_*}$ (note how square root pops-up naturally here) and using Taylor's formula, we get
    \begin{align*}
        \delta_{t+1}^2 = \delta_t^2 - \sqrt{2}\gamma \delta_{t} + \frac{\gamma^2}{2} = \left(\delta_t - \frac{\gamma}{\sqrt{2}}\right)^2\enspace.
    \end{align*}
    We can immediately rule out the approach of taking square roots from both sides since the resulting recursion is not converging.
    Proceeding as with non-smooth convex optimization and re-arranging the above, followed by summation for $t = 1, \ldots, T$, we get
    \begin{align*}
        \sqrt{2} \sum_{t = 1}^T \delta_{t} = \frac{\delta_1^2 - \delta_{T+1}^2}{\gamma} + \frac{\gamma T}{2}\enspace.
    \end{align*}
    For any fixed choice of $\gamma$, there is no linear convergence for the above recursion.
\end{example}
While this is not a proof of the absence of linear rates for SignSVRG, this argument is still plausible since the whole goal of variance reduction techniques is to recover rates that are similar to the non-stochastic setup. Thus, since~\ref{eq:def_signGD} does not converge linearly in general it is hard to expect linear convergence from the SignSVRG.

\subsection{Connection with NoisySIGN of \cite{chen2020distributed}.}
It seems that the first instance of incorporating additive noise under the $\sign$ operator is due to~\cite{chen2020distributed}. Unlike our work, they provide theory for arbitrary additive zero-mean symmetric noise, requiring the variance of the noise to explode with the number of iterations, effectively going to infinity when $T \rightarrow \infty$, to ensure convergence. Interestingly, they have already observed that, in practice, the constant noise largely suffices and there is no need to push its variance to infinity. Furthermore, it seems that the observation of Section~\ref{sec:trick} was simply overlooked by the authors, possibly, due to the generality of the noise. Indeed, their Theorem~5 explicitly assumed that the gradient is bounded by some $Q$ along the trajectory of their algorithm. The main novelty of our work is the use of this trick in the context of variance reduction, leading to faster rates of convergence than those of \cite{chen2020distributed} as well as handling the convex case, all under milder assumptions. Furthermore, we single-out an explicit noise whose magnitude goes to zero as $T \rightarrow \infty$.

\section{Conclusion}
We have presented a simple practical way to incorporate variance reduction technique into the~\ref{eq:signSGD_intro} procedure, fixing its non-convergence issues on finite sums. The trick is based on corruption of the $\sign$ operator by uniform noise.
We have first showed that that this simple approach leads to convergence rates for sums of Lipschitz convex functions.
Then, we proposed two variants of SignSVRG---the first one (\Varone) enjoys convergence rates on non-convex and convex smooth objectives, while the second one (\Vartwo) exhibits rates that are very similar to the full~\ref{eq:def_signGD} algorithm. Our proofs are significantly simpler than typical strategies for variance reduction methods.\\
This work is the first step towards incorporation of variance reduction techniques into~\ref{eq:signSGD_intro} and we expect many more developments in the future. While our work is purely theoretical, it would be beneficial to conduct large-scale experimental studies that we plan to do in future contributions.
The main limitation of the work is the non-adaptive nature of the introduced trick---we need to known either the Lipschitz constant of the gradient or a uniform bound on the gradient. Finally, we are working on extensions of our approach to the setup of Federated Learning, where our idea might be extended in a seemingly straightforward manner.

\bibliography{biblio.bib}
\bibliographystyle{alpha}

\newpage
\appendix

\section{Convergence of signGD}\label{app:signGD}

Let $f: \bbR^d \rightarrow \bbR$ be differentiable and consider the iterations produced by \ref{eq:def_signGD}:
\begin{equation}\label{algo:sign_GD}
    x_{t+1} = x_t - \gamma \sign(\nabla f(x_t))\, .
\end{equation}
We assume throughout this section that there are $p,q \geq 1$, H\"older conjugates, such that for all $x, y \in \bbR^d$, $\norm{\nabla f(x) - \nabla f(y)}_{p} \leq L_q \norm{x-y}_q$.
\begin{proposition}\label{pr:sign_gd}
Assume that $f$ is lower bounded by $f_* \in \bbR$.
For $T \geq 1$ and $x_1 \in \bbR^d$, iterates produced by~\eqref{algo:sign_GD} satisfy
\begin{equation*}
    \Exp[\norm{\nabla f(\bar{x}_T)}_1] \leq \frac{f(x_1) - f_*}{T \gamma} + \frac{L_q\gamma d^{2/q}}{2} \enspace,
\end{equation*}
where $\bar{x}_T$ distributed uniformly on $\{x_1, \ldots, x_T\}$.
\end{proposition}
\begin{proof}
By Lemma~\ref{lm:descent}
    \begin{equation*}
    \begin{split}
         f(x_{t+1})&\leq f(x_t) - \gamma \scalar{\nabla f(x_t)}{\sign(\nabla f(x_t)}  + \frac{\gamma^2 L_q d^{2/q}}{2}\\
         &\leq f(x_t) - \gamma \norm{\nabla f(x_t)}_1 + \frac{\gamma^2 L_q d^{2/q}}{2} \enspace .
    \end{split}
    \end{equation*}
    Thus, 
    \begin{equation}\label{eq:sum_sign}
        \gamma \sum_{t=1}^{T} \norm{\nabla f(x_t)}_1 \leq f(x_1) - f(x_{T+1}) + \frac{\gamma^2 L_q d^{2/q}}{2}\, ,
    \end{equation}
    and 
    \begin{equation*}
        \Exp[\norm{\nabla f(\bar{x}_T)}_1] \leq \frac{f(x_1) - f_*}{T \gamma} + \frac{\gamma L_q d^{2/q}}{2}\enspace.
    \end{equation*}
\end{proof}

\begin{corollary}
 In the setting of Proposition~\ref{pr:sign_gd} and choosing $\gamma =\frac{1}{d^{1/q}} \sqrt{\frac{2}{TL_q}}$, we obtain
\begin{equation*}
    \norm{\nabla f(\bar{x}_T)}_1 \leq d^{1/q}\sqrt{\frac{L_q}{2T}}(f(x_1) - f_* + 1) \enspace .
\end{equation*}
In particular, to produce a point such that $\norm{\nabla f(x_t)}_1 \leq \varepsilon^{-1}$, signGD requires $\cO(\varepsilon^{-2})$ iterations. 
\end{corollary}

Under the first impression, the convergence rate of $\cO(\varepsilon^{-2})$ matches the one of the deterministic gradient descent~\cite{nesterov2018lectures}. However, we see that such rate is obtained for different reasons. Indeed, in the case of gradient descent (for small enough $\gamma$), Eq.~\eqref{eq:sum_sign} is replaced by:
\begin{equation}\label{eq:sum_GD}
    \gamma \sum_{t=1}^T \norm{\nabla f(x_t)}^2 \leq f(x_1) - f(x_{T+1}) \enspace .
\end{equation}
This equation, is in some sense better, since it shows that $f$ decreases along the iterates and implies that $\norm{\nabla f(x_t)} \rightarrow 0$ (as soon as $\gamma$ is small enough). Certainly, signGD will not converge (with a constant step size) to a critical point (indeed, in the simple example where $f(x) = x^2$, signGD will be oscillating around $0$). However, the fact that the squared norm in Eq.~\eqref{eq:sum_GD} is replaced by just a norm (without a square) in Eq.~\eqref{eq:sum_sign}, allows us to recover the same rate of convergence by appropriately choosing the step size. This is reminiscent of what happens in the case of SGD, where the equivalent of Eq.~\eqref{eq:sum_sign} and Eq.~\eqref{eq:sum_GD} is 
\begin{equation}\label{eq:sum_SGD}
    \gamma \sum_{t=1}^T \Exp [\norm{\nabla f(x_t)}^2] \leq f(x_1) - f_* + \frac{L_q}{2} \gamma^2 \sigma^2 \enspace ,
\end{equation}
with $\sigma^2$ being the variance of the stochastic gradient. For similar reasons, SGD is able to produce a point such that $\Exp[\nabla f(x)] = \cO(\varepsilon)$ in $\cO(\varepsilon^{-4})$ iterations. Since, $\norm{\cdot}^2$ in Eq.~\eqref{eq:sum_SGD} is replaced by $\norm{\cdot}_1$ in Eq.~\eqref{eq:sum_sign}, we are able to increase this rate and recover the rates of GD.

\newpage
\section{Remaining proofs}\label{app:proofs}

\begin{proof}[Proof of Lemma~\ref{lm:descent}]
    Recall that for all $u, v\in \bbR^d$ and $p, q \geq 1$, H\"older conjugates, $\scalar{v}{u} \leq \norm{v}_p \norm{u}_q$.
    Since for all $y, x \in \bbR^d$, $\norm{\nabla f(y) - \nabla f(x)}_{p} \leq L_q \norm{y-x}_q$, it holds that
    \begin{equation*}
    \begin{split}
        f(y) - f(x) &= \int_{0}^1 \scalar{f(x + t(y-x))}{y-x} \d t  \\
                &= \scalar{\nabla f(x)}{y-x} + \int_{0}^{1}\scalar{\nabla f(x + t(y-x)) - \nabla f(x)}{y-x}\d t \\
        &=\scalar{\nabla f(x)}{y-x} + \norm{y-x}_q\int_{0}^1 t \norm{\nabla f(x + t(y-x)) - \nabla f(x)}_p \d t \\
        &\leq \scalar{\nabla f(x)}{y-x} + L_q\norm{y-x}_q^2 \int_{0}^1 t \d t \\
        &= \scalar{\nabla f(x)}{y-x} + \frac{L_q}{2}\norm{y-x}_q^2 \enspace .
    \end{split}
    \end{equation*}
\end{proof}

\begin{proof}[Proof of Lemma~\ref{lm:gap_lbound}]
    By Lemma~\ref{lm:descent} we obtain
    \begin{equation*}
    \begin{split}
               f_* - f(x) &\leq \min_{y \in \bbR^d} \left( \scalar{\nabla f(x)}{y-x} + \frac{L_q}{2}\norm{y-x}_q^2 \right)\\
               &= \min_{t \in \bbR}  \min_{\norm{v}_q =1} \left(t\scalar{\nabla f(x)}{v} + \frac{L_q}{2}t^2\right) \\
               &= \min_{t \in \bbR} \left(-t \norm{\nabla f(x)}_p + \frac{L_q}{2}t^2\right)\\
               &= - \frac{\norm{\nabla f(x)}_p^2}{2L_q} \, ,
    \end{split}
    \end{equation*}
    which completes the proof.
\end{proof}

\begin{proof}[Proof of Corollary~\ref{cor:non_cvx}]
For \Varone, starting from Theorem~\ref{thm:main_non_cvx}, ours choice of $\gamma$ yields
\begin{align*}
    \frac{1}{T}\Exp\sum_{t = 1}^T\frac{\norm{\nabla f(x_t)}^2}{2L_{q} D  + \norm{\nabla f(x_t)}_{p}} \leq \frac{d^{1/q}\sqrt{L_q}(f(x_1)  - f_* + 1)}{\sqrt{2T}} \enspace.
\end{align*}
 Thus, setting $\bar{x}_T$ a point randomly sampled along the trajectory, it holds that:
 \begin{align*}
     \Exp\left[\frac{\norm{\nabla f(\bar{x}_T)}^2}{2L_{q} D  + \norm{\nabla f(\bar{x}_T)}_{p}}\right] \leq \frac{d^{1/q}\sqrt{L_q}(f(x_1)  - f_* + 1)}{\sqrt{2T}}\enspace.
 \end{align*}
 Note that the function $(x, y) \mapsto x^2 / (y + b)$ is convex for all $b > 0$ and $x > 0, y > 0$. Therefore, we obtain by Jensen's inequality that
 \begin{align*}
     \frac{\Exp[\norm{\nabla f(\bar{x}_T)}]^2}{2L_{q} D + \Exp[\norm{\nabla f(\bar{x}_T)}_{p}]} \leq \frac{d^{1/q}\sqrt{L_q}(f(x_1)  - f(x_*) + 1)}{\sqrt{2T}}\enspace.
 \end{align*}
This implies that either
 \begin{equation*}
     \Exp[\norm{\nabla f(\bar{x}_T)}_{p}] \leq 2L_{q} D \quad\text{or}\quad \frac{\Exp[\norm{\nabla f(\bar{x}_T)}]^2}{\Exp[\norm{\nabla f(\bar{x}_T)}_{p}]} \leq \frac{2d^{1/q}\sqrt{L_q}(f(x_1)  - f(x_*) + 1)}{\sqrt{2T}}\,.
 \end{equation*}

 For \Vartwo, using Theorem~\ref{thm:main_non_cvx}, we obtain from the same arguments:
 \begin{equation*}
     \frac{\frac{1}{d}\Exp[\norm{\nabla f(\bar x_T)}_1]^2}{2 L_q D + \frac{1}{d}\Exp[\norm{\nabla f(\bar x_T)}_1]} \leq \frac{d^{1/q}\sqrt{L_q}(f(x_1)  - f_* + 1)}{\sqrt{2T}}\enspace,
 \end{equation*}
 and we conclude in the same manner.
\end{proof}

\begin{proof}[Proof of Theorem~\ref{thm:main_cvx}]
Notice that for all $t$ and for all $i \in \{1, \dots, n\}$, any coordinate of $|\nabla f_i(x_t) - \nabla f_i(\tilde x_{k(t)}) + g_{k(t)}|$ is almost surely bounded (conditionally on $x_t$) by $G_t^1$. Thus, denoting $H_t$ the sigma-algebra generated by $x_1, \ldots, x_t$,
\begin{equation*}
    \Exp[x_{t+1} - x_t \mid H_t] = - \gamma \frac{\nabla f(x_t)}{G_t}\, .
\end{equation*}
    In particular, using the convexity of $f$ and the fact $G_t \leq 2L_q D + \|\nabla f(x_t)\|_{p}$, we obtain
    \begin{equation*}
    \begin{split}
         \Exp[\norm{x_{t+1} - x_*}^2 \mid H_t ] &= \norm{x_t - \xs}^2+ 2 \Exp[\scalar{x_{t+1} - x_t}{x_t - \xs} \mid H_t]+ \gamma^2 d \\
         &= \norm{x_t - \xs}^2-2\gamma \frac{\scalar{\nabla f(x_t)}{x_t - \xs}}{G_t} + \gamma^2 d \\
         &\leq \norm{x_t - \xs}^2 +2 \gamma \frac{f_{*} - f(x_t)}{2L_q D + \|\nabla f(x_t)\|_{p}} + \gamma^2 d \, .
    \end{split}
    \end{equation*}
    Rearranging and summing this inequality, we obtain
\begin{equation*}
    \frac{1}{T} \sum_{t=1}^{T} \Exp \left[\frac{ f(x_t)- f_{*}}{2L_q D + \|\nabla f(x_t)\|_{p}} \right] \leq \frac{\norm{x_1 - \xs}^2}{2T \gamma} + \frac{\gamma d}{2} \, .
\end{equation*}
The claimed inequality follows from Lemma~\ref{lm:gap_lbound}.
\end{proof}

\begin{proof}[Proof of Corollary~\ref{cor:cvx}]
    Substituting our choice of $\gamma$ into Theorem~\ref{thm:main_cvx}, we obtain
    \begin{equation*}
       \Exp\left[ \frac{f(\bar x_T) - f_*}{2L_q D + \sqrt{2 L_q (f(\bar x_T) -f_*)}}\right]
       \leq \frac{1}{2}\sqrt{\frac{d}{T}}\left( \frac{\norm{x_1 - \xs}^2}{\alpha} + \alpha \right) \enspace.
    \end{equation*}
    Furthermore, thanks to the convexity of $x \mapsto x^2 / (x + b)$ for all $b > 0$ and $x > 0$, Jensen's inequality yields
    \begin{equation*}
        \frac{ \Exp[\sqrt{f(\bar x_T) - f_*}]^2}{2L_q D + \Exp[\sqrt{2 L_q (f(\bar x_T) - f_*)}]}
        \leq \frac{1}{2}\sqrt{\frac{d}{T}}\left( \frac{\norm{x_1 - \xs}^2}{\alpha} + \alpha \right) \enspace.
    \end{equation*}
    Finally, using the expression for $D$, we have either
    \begin{equation*}
       \Exp[\sqrt{f(\bar x_T) - f_*}] \leq
       \sfP\sqrt{\frac{2L_q}{T}}\qquad\text{or}\qquad
\Exp[\sqrt{f(\bar{x}_T) - f_*}] \leq \sqrt{\frac{2 L_qd}{T}} \left( \frac{\norm{x_1 - \xs}^2}{\alpha} + \alpha\right)\enspace .
    \end{equation*}
    The proof is concluded.
\end{proof}

\end{document}